\documentclass[12 pt]{article}
\usepackage{fullpage}
\usepackage{amssymb,amsmath,amsfonts, amscd}
\usepackage{array}
\usepackage{amsmath}
\usepackage{amssymb}

\begin{document}

\newtheorem{theorem}{Theorem} [section]
\newtheorem{conjecture}[theorem]{Conjecture}
\newtheorem{tha}{Theorem}
\renewcommand{\thetha}{\Alph{tha}}
\newtheorem{corollary}[theorem]{Corollary}
\newtheorem{lemma}[theorem]{Lemma}
\newtheorem{proposition}[theorem]{Proposition}
\newtheorem{construction}[theorem]{Construction}
\newtheorem{question}[theorem]{Question}
\newtheorem{definition}[theorem]{Definition}
\newtheorem{observation}{Observation}
\newtheorem{remark}[theorem]{Remark}
\newtheorem{fact}[theorem]{Fact}
\setlength{\textwidth}{17cm}
\setlength{\oddsidemargin}{-0.1 in}
\setlength{\evensidemargin}{-0.1 in}
\setlength{\topmargin}{0.0 in}

\def \df {\noindent {\bf Definition. }}

\newcommand{\dist}{{\rm dist}}
\newcommand{\Dist}{{\rm Dist}}
\newcommand{\chib}{\chi_B}
\newcommand{\Forb}{{\rm Forb}}

\newcommand{\qedbox}{$\blacksquare$ \newline}
\newenvironment{proof}%
{%
\noindent{\it Proof.} } {
\hfill\qedbox }

\newcommand{\proofend}{\hfill\qedbox}

\def\qed{\hskip 1.3em\hfill\rule{6pt}{6pt} \vskip 20pt}

\linespread{1.0}
\input epsf
\def\epsfsize#1#2{0.5#1\relax}
\def\O{\text{O}}
\def\o{\text{o}}
\def\ex{\text{ex}}
\def\Z{\mathbb Z}

\def\fl#1{\lfloor #1 \rfloor}
\def\ce#1{\lceil #1 \rceil}

\def \cH{{\cal H }}
\def \cF{{\cal F}}
\def \cG{{\cal G}}
\def \cQ{{\cal Q}}
\def \cA{{\cal A}}
\def \cD{{\cal D}}

\def \f {{\cal F }}
\def \A {{\cal A}}
\def \D {{\cal D}}
\def \fn2 {{\lfloor n/2 \rfloor}}
\def \cn2 {{\lceil  n/2 \rceil}}

\newcommand{\bin}[2]{{#1\choose #2}}
\newcommand{\comp}{\overline}
\newcommand{\exval}{{\rm ex}}

\title{A note on short cycles  in a hypercube}
\author{Maria Axenovich\thanks{Department of Mathematics, Iowa
State University, Ames, IA 50011, {\tt
axenovic@math.iastate.edu}}\and Ryan Martin\thanks{%
Department of
Mathematics, Iowa State University, Ames, IA 50011, {\tt
rymartin@iastate.edu}}}

\date{}
\maketitle

\begin{abstract}
How many edges can a quadrilateral-free subgraph of a hypercube
have? This question was raised by Paul Erd\H{o}s about $27$ years
ago. His conjecture that such a subgraph asymptotically  has at most
half the edges of a hypercube is still unresolved. Let  $f(n,C_l)$
be the largest number of edges in a subgraph of a  hypercube $Q_n$
containing no cycle of length $l$. It is known that $f(n, C_l) = o
(|E(Q_n)|)$, when $l= 4k$, $k\geq 2$ and that $f(n, C_6) \geq
\frac{1}{3} |E(Q_n)|$. It is an open question
to determine $f(n, C_l)$  for $l=4k+2$, $k\geq 2$.  Here,
we give a general upper bound for $f(n,C_l)$ when $l=4k+2$ and
provide  a coloring of $E(Q_n)$ by $4$ colors containing no
induced monochromatic $C_{10}$.

\end{abstract}

\section{Introduction}
Let $Q_n$ be a hypercube of dimension $n$. We treat its vertices
as binary sequences of length $n$ or as subsets of a set $[n]=\{1,
2, \ldots, n\}$, whichever  is more convenient. The edges  of $Q_n$
correspond to pairs of sets with symmetric difference of size $1$
or, equivalently, to pairs of sequences with Hamming distance $1$.
We denote the set of subsets of $[n]$  of size $k$ by $\binom{[n]}{k}$. We
say that a subset of edges forms an $i^{\rm th}$ edge layer of a
hypercube if these  edges join vertices in $\binom{[n]}{i}$ and
$\binom {[n]}{i+1}$. It is a classical question to find a dense
subgraph of a hypercube without cycles of a certain fixed length.
Moreover, it is of interest to consider the Ramsey properties of
cycles in $Q_n$. We say that a cycle $C_l$ has a {\it Ramsey
property} in $Q_n$, if for a every  $k$, there is an
$N_0$ such that if $n>N_0$ and the edges of $Q_n$ are colored in
$k$ colors then there is always a monochromatic $C_l$ in a such
coloring.

It is easy to see that there is a coloring of $E(Q_n)$ in two
colors with no monochromatic $C_4$. Indeed, color all edges in
each edge-layer in the same color, using two alternating colors on
successive layers. Each color class is a  quadrilateral-free subgraph of
$Q_n$, with the larger one  having at least $\frac{1}{2} |E(Q_n)|$ edges. There is a
coloring by Conder \cite{Co} using three colors on $E(Q_n)$ and
containing no monochromatic  $C_6$, in particular providing a
hexagon-free subgraph of $Q_n$ with at least $|E(Q_n)|/3$ edges.
This shows that both $C_4$ and $C_6$ do not have Ramsey property
in a hypercube.

In general, if $f(n, C_l)$ is the largest number of edges in a
subgraph of $Q_n$ with no cycle of length $l$ then the following
facts are known: Fan Chung \cite{C} proved that $f(n, C_4)\leq .623|E(Q_n)|$. The conjecture of Erd\H{o}s that $f(n, C_4) =
(1/2+o(1))|E(Q_n)|$ is still open. Chung, \cite{C} also gave the
following upper bound for cycles of length $0$ modulo $4$: $f(n,
C_{4k})\leq  cn^{(1/2k)-(1/2)}|E(Q_n)|$, $k\geq 2$.

Thus, $C_{4k}$ has the Ramsey property for $k\geq 2$. In \cite{C},
Chung also raised the question about Ramsey properties of cycles
$C_{4k+2}$, $k\geq 2$, in particular about $C_{10}$. Recently,
Alon, Radoi\u{c}i\'{c}, Sudakov and Vondr\'{a}k, \cite{ARSV},
completely settled the problem about Ramsey properties of cycles
in the hypercube by proving that for any $l\geq 5$, $C_{2l}$ has a
Ramsey property. In this note, we investigate the Ramsey
properties of induced cycles in the hypercube and show that
$C_{10}$, as an induced subgraph of $Q_n$, does not have the
Ramsey property.

\begin{theorem}\label{construction}
There is a  coloring  of $E(Q_n)$ using $4$ colors  such that there is no induced monochromatic
$C_4$, $C_6$, or $C_{10}$.
\end{theorem}

We prove this theorem and  describe some properties of the corresponding  coloring
which are of independent interest in Section \ref{proof}.
For completeness, we  provide a general upper
bound on the maximum number of edges in a subgraph of a hypercube
containing no cycle of length $C_{4k+2}$ in Section \ref{UB}.  For the general
graph-theoretic definitions we refer the reader to \cite{W}.

\section{A coloring with no induced monochromatic  $C_{10}$}\label{proof}

For a binary sequence (or binary word) $x$, let $w(x)$ be the weight (or number of
$1$'s) in $x$.  If the vertices corresponding to the binary
sequences $x$ and $y$ form an edge in $Q_n$ (i.e., they differ in
exactly one position) let the common substring that precedes the
change be called the common {\it prefix} and be denoted $p=p(xy)$.

\noindent
{\bf Description of a coloring $c$}\\
\begin{quote}
For an edge $xy$, with $x\in\bin{[n]}{k}$, $y\in\bin{[n]}{k+1}$
and having  common prefix $p=p(xy)$,   let $c(xy)=(c_1,c_2)$, where
$$ c_1\equiv k\pmod{2},\qquad c_2\equiv w(p)\pmod{2} . $$
\end{quote}

\subsection{Proof of Theorem \ref{construction}}
We shall prove that the coloring $c$ does not produce induced monochromatic cycles of length $4$, $6$, or $10$.
Since the sets of colors used on consecutive edge-layers of $Q_n$
are disjoint, the monochromatic cycles can occur only within
edge-layers. There is no $C_4$ within an edge-layer, so there is
no monochromatic $C_4$.

Let $C$ be a monochromatic cycle in $Q_n$ under coloring $c$. Then
$C=x_1,y_1, \ldots, x_m,y_m, x_1$, where $x_i \in \binom{[n]}{k}$,
$y_i \in \binom{[n]}{k+1}$, for $i=1, \ldots, m$ and for some $k\in
\{1, \ldots, n-1\}$. Note also that $y_i=x_i\cup x_{i+1}$, for $i=1,
\ldots, m$ where addition is taken modulo $m$. In particular, we
have that ${\rm dist}(x_i, x_{i+1})=2$ and ${\rm
dist}(y_i,y_{i+1})=2$, for  $i=1, \ldots, m$.

In the following lemma, we settle the case with $C_6$ mostly to
illustrate the properties of the coloring.

\begin{lemma}
There is no monochromatic $C_6$ in $Q_n$ under coloring $c$.
\end{lemma}

\begin{proof}
Let $C=x_1y_1x_2y_2x_3y_3x_1$ be a monochromatic $C_6$. Then
$x_i$'s can be written as  $a_0 1 a_1 0a_2 0a_3$, $a_0 0 a_1 1a_2
0a_3$, $a_0 0 a_1 0a_2 1a_3$, for some binary words $a_0, \ldots,
a_3$. The corresponding $y_i$'s are $a_0 1 a_1 1 a_2 0a_3$, $a_0 0
a_1 1  a_2 1 a_3$, $a_0 1 a_1 0a_2 1 a_3$. Therefore, the set of
common prefixes on the edges of $C$ contains $a_0 0a_1$ and $a_0 1
a_1$. These two prefixes have different weights modulo $2$, thus
$C$ is not monochromatic under coloring $c$.
\end{proof}
\noindent
For a word ${\bf u}=u_1u_2\cdots u_n$, let the reverse of
${\bf u}$  be  ${\overline{{\bf u}}}=u_nu_{n-1}\cdots u_1$. \\

\noindent {\it Observation 1.} Let $x, x' \in \binom{[n]}{k}$ and
$y, y' \in \binom{[n]}{k+1}$. Let $xy$, $x'y'$  be  edges of $Q_n$
such that $x$ and $y$  have  a common prefix $p$ and a common
suffix $s$ and $x'$ and $y'$  have a common prefix $p'$ and a
common suffix $s'$. Then $w(p)\equiv  w(p')\pmod{2}$ iff
$w(s)\equiv w(s')\pmod{2}$. This is also equivalent to saying that
if $q$ is the common prefix of $\overline{x}$ and $\overline{y}$
and $q'$ is the common prefix of $\overline{x'}$ and
$\overline{y'}$ then $w(p)\equiv w(p')\pmod{2}$ iff $w(q)\equiv
w(q')\pmod{2}$.\\

\noindent
For a binary sequence $x$ of length $n$, we denote by $x[I]$ its restriction to positions from $I$.
For example, if $x= 0010010110$ and $I=\{2,3, 5, 6, 10\}$ then
$x[I]=01010$.\\

\noindent {\bf Main idea of the proof.} Let $C=x_1, y_1, \ldots,
x_5, y_5, x_1$ be an induced  $C_{10}$ in $Q_n$ with $x_i \in
\binom{[n]}{k}$ and $ y_i \in \binom{[n]}{k+1}$, $i=1, \ldots, 5$.
Let $I(C)=\{i_1, \ldots, i_m\}$, $i_1<i_2<\cdots <i_m$,  be the set of positions
where $x_i$ and  $x_{i+1}$ differ for some $i\in [5]$.

We define a  $5\times m$ matrix ${\bf
A}={\bf A}(C)$ where the $i^{\rm th}$ row of ${\bf A}$ is  $x_i[I]$.

 First, we shall show that $m=5$. Second, we shall show that none of the $5\times 5$
binary matrices are possible as ${\bf A}(C)$ for an induced cycle
$C$ monochromatic under coloring $c$.\\

\noindent
{\bf Note.} For the rest of the proof we assume that $C$ is monochromatic under coloring $c$ and
that it is an induced cycle in $Q_n$. We take the addition modulo $5$ unless otherwise specified.\\

\begin{lemma}\label{dif_pos}
$|I(C)|=5$. Moreover, $1$'s occur in consecutive positions (modulo $5$)  in each column of ${\bf A}(C)$.
\end{lemma}

\begin{proof}
Let $|I(C)|=\{i_1,\ldots,i_m\}$.
Observe first that all five  rows of ${\bf A}(C)$ form distinct binary words of the same weight.

\noindent {\it Case 1.} Let $m\leq 3$. Then there are at most $3$ distinct binary words of
the same weight with length $m$, a contradiction.

\noindent {\it Case 2.} Let  $m= 4$. If the number of $1$'s in each row is $i$,
we have that the total possible number of distinct rows is $\binom{4}{i}$. This number is at least  $5$ only  if $i=2$.
Thus each row of $A$ has two $1$'s.
Therefore, $y_i[I]$ has weight three for each $i=1, \ldots, 5$.
But there are only four binary words of length four and weight three,
a contradiction to the fact that we have five distinct $y_i$'s.

\noindent {\it Case 3.} Assume that $m\geq 5$. We shall introduce
a (multi-)graph $H$ on vertices $v_{i_1}, \ldots, v_{i_m}$.
Let $v_{i_s}v_{i_t}\in E(H)$ if some pair of consecutive $x_i$'s differ exactly in positions $i_s$ and
$i_t$. Note that since $C$ is a cycle, a graph $H$ must have even
nonzero degrees, thus $5=|E(H)| \geq |V(H)|=m$. 
Hence, $m=5$ and  each vertex has degree $2$. 
Since a vertex of $H$ corresponds to a column of
${\bf A}(C)$, and the  degree of a vertex in $H$ corresponds to a number of times $1$ changes to $0$ and $0$ changes to $1$ on consecutive elements in a corresponding column of ${\bf A}(C)$ (in cyclic order), we must have  $1$'s occur consecutively in each column
of ${\bf A}(C)$.

\end{proof}
\vskip .2in

We say that a matrix ${\bf A}$ has {\bf  bad  prefixes} if there are
two pairs of consecutive rows ${\bf r}_i,{\bf r}_{i'}$ and ${\bf
r}_{j}, {\bf r}_{j'}$ in ${\bf A}$, $i'=i\pm 1\pmod 5$, $j'=j \pm 1\pmod 5$,  such that either \vspace{-.1in}
\begin{itemize}
\item[(1)] ${\bf r}_i$ starts with ${\bf a}0{\bf b} 0$, ${\bf r}_{i'}$
starts with ${\bf a}0{\bf b}1$; ${\bf r}_j$ starts with ${\bf a}0{\bf b}1$,
${\bf r}_{j'}$ starts with ${\bf a}1{\bf b}0$ or \vspace{-.1in} \\
\vspace{-.2in} \item[(2)] ${\bf r}_i$  starts with ${\bf a}0{\bf b}0$,
${\bf r}_{i'}$ starts with ${\bf a}0{\bf b}1$; ${\bf r}_j$  starts with
${\bf a}1{\bf b}1$, ${\bf r}_{j'}$ starts with ${\bf a}1{\bf b}0$,
\vspace{-.1in}
\end{itemize}
for some binary words ${\bf a}$, ${\bf b}$. \\

For a matrix ${\bf A}$ with rows ${\bf v}_1, \ldots, {\bf v}_5$, let
$\overline{{\bf A}}$ be a matrix with rows $ \overline{{\bf v}_1},
\ldots, \overline{{\bf v}_5}$ in the same order.

\begin{lemma}\label{prefix_suffix}
Let  ${\bf A}={\bf A}(C)$, then neither ${\bf A}$ nor
$\overline{{\bf A}}$ have {\it bad} prefixes.
\end{lemma}

\begin{proof}
Assume that the matrix ${\bf A}$ has bad prefixes.
Then  $x_i= a_1 0 a_2 0 a_3$,
$x_{i'} = a_1 0 a_2 1a_3'$, and $y= x_i \cup x_{i'} = a_1 0 a_2 1 a_3$, for some binary words $a_1, a_2,
a_3, a_3'$. Here, $y=y_i$ or $y=y_{i'}$.  The common prefix of $x_i$ and $y$ is $a_1 0 a_2$.

In part (1) we also have that  $x_j= a_1 0 a_2 1 b_3$, $x_{j'} = a_1 1 a_2 0 b_3$,
$y'= x_j \cup x_{j'}=a_1 1 a_2 1 b_3$, for some binary word $b_3$. Here $y'=y_j$ or $y'=y_{j'}$.  The common prefix of $x_{j'}$ and
$y'$ is $a_1 1 a_2$. Thus, $c(x_{j'}y') \neq c(x_{i}y)$ since the weights of
corresponding prefixes are different modulo $2$.

In part (2) we also have that  $x_j= a_1 1 a_2 1 b_3$, $x_{j'} = a_1 1 a_2 0 b_3'$,
$y'=x_j\cup x_{j'}= a_1 1 a_2 1 b_3'$, for some binary words $b_3, b_3'$. Here $y'=y_j$ or $y'=y_{j'}$.
  The common prefix of $x_{j'}$ and $y'$ is
$a_1 1 a_2$. Thus, $c(x_{j'}y') \neq c(x_{i}y)$.

In both  cases we have a  contradiction
to the assumption that the  cycle $C$ is monochromatic.

When $\overline{{\bf A}}$ has
bad prefixes, we use Observation $1$ to arrive at the same
conclusion.
\end{proof}

\begin{lemma} \label{2ones}
The matrix ${\bf A}={\bf A}(C)$ has exactly two $1$'s in each row.
\end{lemma}

\begin{proof}
Since all  $x_i$'s have  the same weight, each row of ${\bf A}$ contains
the same number of $1$'s.

Assume that each row of ${\bf A}$ contains exactly one $1$. There are
three consecutive rows of ${\bf A}$ which form binary words in
increasing  lexicographic order. It is a routine observation to see that either ${\bf A}$ or ${\bf \overline{A}}$  has
{\it bad} prefixes because of these rows.

Assume that there are three $1$'s and two $0$'s in each row of
${\bf A}$. Then  $y_i[I]$  has weight $4$, for $i=1, \ldots, 5$. There
are exactly $5$ binary words of length $5$ and weight $4$, namely
$11110, 11101, \ldots, 01111$.
Since each column of ${\bf A}$ has consecutive $1$'s, it has consecutive $0$'s.
If the number of $0$'s in some column, $i$,  is more than two, then
there are at least two words, $y_t[I]$ and $y_q[I]$ which have
$0$ in the $i$th position, a contradiction.
Thus each column of ${\bf A}$ must have exactly two consecutive $0$'s and three
consecutive $1$'s (in cyclic order). Then each $x_i= a_0 * a_1 *
a_2 * a_3 * a_4 *a_5$, where $a_j$'s are some binary words and
$*\in \{0,1\}$ are located in positions from $I$.

If there is a row with two consecutive $0$'s in the first and
second column then we immediately get {\it bad} prefixes since the
first two columns must be as follows, up to cyclic rotation of
the rows:
\begin{equation}\nonumber
\left(
\begin{matrix}
1&1\\
1&1\\
1&0\\
0&0\\
0&1
\end{matrix}
\right)
\end{equation}

Thus we can assume that the first two columns of ${\bf A}$ or
${\bf \overline {A}}$ are as follows, up to cyclic rotation of the
rows:

\begin{equation}\nonumber
\left(
\begin{matrix}
1&0\\
1&0\\
1&1\\
0&1\\
0&1
\end{matrix}
\right)
\end{equation}
Now, we see that $a(3,3)=0$, otherwise $a(3,4)=a(3,5)=0$ and we
arrive at {\it bad} prefixes in ${\bf \overline{A}}$, similarly to the previous argument.
Then we have two possible cases for the first three columns of
${\bf A}$:

\begin{equation}\nonumber
\left(
\begin{matrix}
1&0&1\\
1&0&0\\
1&1&0\\
0&1&1\\
0&1&1
\end{matrix}
\right)\quad\mbox{ and }\quad\left(
\begin{matrix}
1&0&1\\
1&0&1\\
1&1&0\\
0&1&0\\
0&1&1
\end{matrix}
\right).
\end{equation}
Note that in the first  case we  have prefixes $a_01a_11a_2$
and $a_01a_10a_2$, which have different weights modulo $2$, and in
the second case we have prefixes $a_01a_11a_2$ and $ a_00a_11a_2$,
which have different weight modulo $2$, each case is a
contradiction to the fact that $C$ is monochromatic.

Finally, the rows of ${\bf A}$ cannot contain more than three  $1$'s
each since in this case $y_i=y_j$ for $i,j\in [5]$.
\end{proof}

Note that Lemma \ref{2ones} implies that the total number of ones
in $\bf A$ is $10$.

\begin{lemma}\label{three_ones}
If there are indices $i_1, i_2, i_3$ and $j_1, j_2$ such
that ${\bf A}(C)$ has entries $a(i_1, j_1)= a(i_1, j_2)=a(i_2,
j_1)=a(i_2, j_2)=a(i_3, j_1)=a(i_3, j_2)=0$, then $C$ has a chord
in $Q_n$.
\end{lemma}

\begin{proof}
Consider $x_{i_1}, x_{i_2}, x_{i_3}$. Note that $x_{i_1}\cup
x_{i_2}= x_{i_2}\cup x_{i_3}=x_{i_1}\cup x_{i_3}$.  At least two
of these three vertices must be successive vertices of $C$ in a
layer $\binom{[n]}{k}$, say without loss of generality that
$x_{i_1}=x_1$ and $x_{i_2}=x_2$. But then $x_{i_3} \neq x_3$,
$x_{i_3}\neq x_5$, otherwise we shall  have $y_1=y_2$ or
$y_1=y_5$. Thus $x_{i_3}=x_4$ and $y_1x_4$ forms a chord of $C$ in
$Q_n$.
\end{proof}

Now, we are ready to complete the proof of the Theorem \ref{construction}.
 Recall first that  Lemma \ref{dif_pos} implies that $1$'s occur consecutively in each column of ${\bf A}$ (up to
cyclic  rotation of the rows). Second, observe that
there are no  two consecutive rows of  ${\bf A}$  such that one has two $0$'s in columns $j,j'$ and another has two $1$'s in the same columns  $j,j'$
(otherwise $x_i$ and $x_{i+1}$  will have a distance at least four between them).
We see
from Lemma \ref{three_ones} that ${\bf A}$ has at most one column
with at most one $1$. Since the total number of $1$'s in ${\bf A}$ is
$10$, we have that  each column of ${\bf A}$ has at at most three $1$'s.
Moreover there is at most one column with three $1$'s. Let  ${\bf A}$
have  exactly one column with exactly one $1$. Then, there must be
a column with exactly three $1$'s.
Consider possible submatrices formed by these two columns (up to cyclically rotating the rows):
\begin{equation}\nonumber
B_1=\left(
\begin{matrix}
0&0\\
1&0\\
0&1\\
0&1\\
0&1
\end{matrix}
\right), \quad
B_2=\left(
\begin{matrix}
0&1\\
1&1\\
0&1\\
0&0\\
0&0
\end{matrix}
\right).
\end{equation}

Since there are exactly two $1$'s in each row of ${\bf A}$
and Hamming distance between consecutive rows is $2$, we have,
in case of $B_1$, that there must  be a column   $i$, where  $i\notin \{1, 2\}$, such that  $a(2,i) =
a(3,i)=1$, which implies that  $a(1,i)=1$.  Thus there are two columns with
exactly three $1$'s, a contradiction.  $B_2$  is possible only if all five columns  {\it  up to permutation}  are as follows.
\begin{equation}\label{eq2}
\left(
\begin{matrix}
0&1& 1& 0&0\\
1&1& 0& 0&0\\
0&1& 0& 1&0\\
0&0& 0& 1&1\\
0&0& 1& 0&1
\end{matrix}
\right)
\end{equation}
(This is true because each row has two
$1$'s and because the  Hamming distance between consecutive rows is $2$.)

If we apply Lemma \ref{prefix_suffix} to all pairs of columns
in (\ref{eq2}), we see that the only possible pairs for the first two and
the last two columns are $\{1, 2\}, \{1,5\}, \{2,5\}$. Thus it is
impossible to choose two acceptable last columns and two
acceptable first columns at the same time.  This concludes the
argument that ${\bf A}$ cannot have columns with exactly one $1$.

Therefore ${\bf A}$ has exactly two consecutive $1$'s in each column.
 Consider the first two columns. There are only two
possibilities: when there is a row $i$ such that $a(i,1)=a(i,2)=1$
and when there is no such row. In both situations, it is again a routine observation to
see that matrix ${\bf A}$ has bad prefixes.

Thus no $5\times 5$ matrix of $0$'s and $1$'s can be equal to ${\bf A}(C)$
for some monochromatic induced $10$-cycle $C$.  This concludes the proof
of  Theorem \ref{construction}. \proofend

\noindent
{\bf Remark.} The above proof basically reduces the analysis of the coloring $c$ on $E(Q_n)$
to the analysis of the coloring $c$ on $E(Q_5)$.
We believe that a similar result should hold  for chordless $C_{4k+2}$ in $Q_n$, $k>2$.

\section{The general results for even cycles in a hypercube}\label{UB}

The following result of Fan Chung \cite{C} states that
dense subgraphs of a hypercube contain all ``not too long''
cycles of lengths divisible by $4$.
\begin{theorem}
For each $t$, there is a constant $c=c(t)$ such that if  $G$ is a subgraph of $Q_n$ with at least $cn^{-1/4}|E(Q_n)|$ edges
then $G$ contains all cycles of lengths $4k$, $2\leq k\leq t$.
\end{theorem}

On the other hand, we can conclude that a dense subgraph of a
hypercube contains a cycle of length $4k+2$ for some $k$ which
follows from the following strengthening of a classical theorem of
Bondy and Simonovits \cite{BS} by Verstra\"ete \cite{V}.

\begin{theorem}\cite{V}
Let $q\geq 2$ be a natural number and $G$ a bipartite graph of
average degree at least $4q$ and girth $g$. Then there exist
cycles of $(g/2-1)q$ consecutive even lengths in $G$.
\end{theorem}

Unfortunately, these results still do not guarantee the existence
of a cycle of length $4k+2$ for some small fixed  $k$ in dense subgraphs of a hypercube.
For completeness, we include a general bound on the maximum number of edges in
 a $C_{4k+2}$-free subgraph of a hypercube.

\begin{theorem}
$f(n, C_{4k+2}) \leq (1+o(1))\frac{1}{\sqrt 2} n2^{n-1},$  $k\geq
1$.
 \end{theorem}

\begin{proof} Let $k\geq 1$ be an integer and let $G$ be a subgraph
of $Q_n$ containing no cycles of length $4k+2$.


Let $d_v$ be the degree of a vertex $v$ in $G$. For each $v\in
V(Q_n)$, we introduce a graph $H_v=(V,E)$ such that $V = \{ u\in
V(Q_n) : uv \in E(Q_n) \}$ and  $E = \{ \{u,w\}: \mbox { there is a }
2-\mbox {path from } u \mbox { to } w \mbox { in } G \mbox {
different from } uvw \}.$

If $uw, u'w'$ are two distinct edges in $H_v$ ($u$ might coincide with $u'$), then
there is a unique $(u-w)$-path $uxw$, $x\neq v$ and a unique $(u'-w')$-path
$u'x'w'$, $x'\neq v$ in $G$, moreover $x\neq x'$. We note first
that $H_v$ does not have $C_{2k+1}$.  To that end, consider $u_1,
\ldots, u_{2k+1}, u_1$, a cycle in $H_v$. The previous observation
implies that there is a cycle $u_1, x_1, u_2, x_2, \ldots,
u_{2k+1}, x_{2k+1}, u_1$ in $G$, a contradiction.  If $n$ is sufficiently large, then
any graph on $n$ vertices with no copy of $C_{2k+1}$ has at most
$n^2/4$ edges, see for example \cite{B}. Hence, $H_v$ has at most $n^2/4$ edges, for $n$
large enough.

Then we have
\begin{equation} \label{H_v}
\sum_{v\in V(Q_n)} |E(H_v)| \leq 2^n n^2/4.
\end{equation}



Next, we count $p(G)$, the number of paths of length two in $G$ in two ways.
Trivially,
$$p(G)= \sum_{v\in V(Q_n)} \binom{d_v}{2}.$$
On the other hand,
$$p(G) = \sum_{v\in V(Q_n)}|E(H_v)|, $$
because if $uv'w$ is a path of length two in $G$ then $\{u,w\}\in E(H_v)$ for the unique $v$, $v\neq v'$,
which is adjacent to both $u$ and $w$ in $Q_n$.
Thus, we have
\begin{equation} \label{sums_eq}
\sum_{v\in V(Q_n)} \binom{d_v}{2} =\sum_{v\in V(Q_n)}|E(H_v)|.
\end{equation}

Using the Cauchy-Schwarz inequality, we have
\begin{eqnarray}
\sum_{v\in V(Q_n)} \binom{d_v}{2} & = &
\sum_{v\in V(Q_n)} (d_v^2/2 - d_v/2) \nonumber \\
& = & \left( \sum_{v\in V(Q_n)} d_v^2/2\right) - |E(G)|
\nonumber \\
& \geq & 2^{-n-1}  \left(\sum_{v\in V(Q_n)} d_v \right)^2 - |E(G)|
\nonumber \\
& = & 2^{-n-1} (2|E(G)|)^2 - |E(G)|. \label{d_v'}
\end{eqnarray}

Combining (\ref{H_v}), (\ref{sums_eq}), and (\ref{d_v'}),  we
have
$$ 2^{-n-1} (2|E(G)|)^2 - |E(G)| \leq 2^n n^2/4.$$
 
If $|E(G)|= an2^n$, then
$a\leq \frac {1}{2n} + \frac{1}{2} \sqrt { 2 + 1/n^2} \leq
1/\sqrt {2}+ \epsilon$, where $\epsilon \rightarrow 0$ as $n \rightarrow \infty$.
\end{proof}

\noindent
{\bf Acknowledgments. } The authors  thank  the referees for
careful reading and helpful suggestions which improved the presentation of the results.

\end{document}